\documentclass[12pt,reqno]{amsart}

\usepackage{color}
\usepackage{amsmath, amsthm, amssymb}
\usepackage{amsfonts}
\usepackage[ansinew]{inputenc}
\usepackage[dvips]{epsfig}
\usepackage{graphicx}
\usepackage[english]{babel}
\usepackage{hyperref}
\theoremstyle{plain}
\newtheorem{thm}{Theorem}[section]

\newtheorem{lem}[thm]{Lemma}
\newtheorem{prop}[thm]{Proposition}

\theoremstyle{definition}

\theoremstyle{remark}
\newtheorem{rem}[thm]{Remark}

\numberwithin{equation}{section}

\newcommand{\average}{{\mathchoice {\kern1ex\vcenter{\hrule height.4pt
width 6pt depth0pt} \kern-9.7pt} {\kern1ex\vcenter{\hrule
height.4pt width 4.3pt depth0pt} \kern-7pt} {} {} }}

\def\R{\mathbb{R}}

\begin{document}

\title{Regularity for the fractional Gelfand problem up to dimension 7}

\author{Xavier Ros-Oton}

\address{Universitat Polit\`ecnica de Catalunya, Departament de Matem\`{a}tica  Aplicada I, Diagonal 647, 08028 Barcelona, Spain}
\email{xavier.ros.oton@upc.edu}

\thanks{The author was supported by grants MINECO MTM2011-27739-C04-01 and GENCAT 2009SGR-345.}


\begin{abstract}
We study the problem $(-\Delta)^su=\lambda e^u$ in a bounded domain $\Omega\subset\R^n$, where $\lambda$ is a positive parameter.
More precisely, we study the regularity of the extremal solution to this problem.

Our main result yields the boundedness of the extremal solution in dimensions $n\leq7$ for all $s\in(0,1)$ whenever $\Omega$ is, for every $i=1,...,n$, convex in the $x_i$-direction and symmetric with respect to $\{x_i=0\}$.
The same holds if $n=8$ and $s\gtrsim0.28206...$, or if $n=9$ and $s\gtrsim0.63237...$.
These results are new even in the unit ball $\Omega=B_1$.
\end{abstract}

\maketitle

\section{Introduction and results}

Let $s\in(0,1)$ and $\Omega$ be a bounded smooth domain in $\R^n$, and consider the problem
\begin{equation}\label{pb-G}
\left\{ \begin{array}{rcll} (-\Delta)^s u &=&\lambda e^u&\textrm{in }\Omega \\
u&=&0&\textrm{in }\mathbb R^n\backslash \Omega.
\end{array}\right.\end{equation}
Here, $\lambda$ is a positive parameter and $(-\Delta)^s$ is the fractional Laplacian, defined by
\begin{equation}
\label{laps}(-\Delta)^s u(x)= c_{n,s}{\rm PV}\int_{\R^n}\frac{u(x)-u(y)}{|x-y|^{n+2s}}dy.
\end{equation}
The aim of this paper is to study the regularity of the so-called extremal solution of the problem \eqref{pb-G}.

For the Laplacian $-\Delta$ (which corresponds to $s=1$) this problem is frequently called the Gelfand problem \cite{Gelfand}, and the existence and regularity properties of its solutions are by now quite well understood \cite{Liouville,JL,NS,MP,CR}; see also \cite{Fujita-exp,Peral-Vaz}.

Indeed, when $s=1$ one can show that there exists a finite extremal parameter $\lambda^*$ such that if $0<\lambda<\lambda^*$ then it
admits a minimal classical solution $u_\lambda$, while for $\lambda>\lambda^*$ it has no weak solution.
Moreover, the pointwise limit $u^*=\lim_{\lambda\uparrow \lambda^*}u_\lambda$ is a weak solution of problem with $\lambda=\lambda^*$.
It is called the extremal solution.
All the solutions $u_\lambda$ and $u^*$ are stable solutions.

On the other hand, the existence of other solutions for $\lambda<\lambda^*$ is a more delicate question, which depends strongly on the regularity of the extremal solution $u^*$.
More precisely, it depends on the boundedness of $u^*$.

It turns out that the extremal solution $u^*$ is bounded in dimensions $n\leq9$ for any domain $\Omega$ \cite{MP,CR}, while $u^*(x)=\log\frac{1}{|x|^2}$ is the (singular) extremal solution in the unit ball when $n\geq10$.
This result strongly relies on the stability of $u^*$.
In the case $\Omega=B_1$, the classification of all radial solutions to this problem was done in \cite{Liouville} for $n=2$, and in \cite{JL,NS} for $n\geq3$.

For more general nonlinearities $f(u)$ the regularity of extremal solutions is only well understood when $\Omega=B_1$.
As in the exponential case, all extremal solutions are bounded in dimensions $n\leq 9$, and may be singular if $n\geq10$ \cite{CC}.
For general domains $\Omega$ the problem is still not completely understood, and the best result in that direction states that all extremal solutions are bounded in dimensions $n\leq4$ \cite{C4,V}.
In domains of double revolution, all extremal solutions are bounded in dimensions $n\leq7$ \cite{CR}.
For more information on this problem, see \cite{BV} and the monograph \cite{D}.

For the fractional Laplacian, the problem was studied by J. Serra and the author \cite{RS-extremal} for general nonlinearities $f$.
We showed that there exists a parameter $\lambda^*$ such that for $0<\lambda<\lambda^*$ there is a branch of minimal solutions $u_\lambda$, for $\lambda>\lambda^*$ there is no bounded solutions, and for $\lambda=\lambda^*$ one has the extremal solution $u^*$, which is a stable solution.
Moreover, depending on the nonlinearity $f$ and on $n$ and $s$, we obtained $L^\infty$ and $H^s$ estimates for the extremal solution in general domains $\Omega$.
Note that, as in the case $s=1$, once we know that $u^*$ is bounded then it follows that it is a classical solution; see for example \cite{RS-Dir}.

For the exponential nonlinearity $f(u)=e^u$, our results in \cite{RS-extremal} yield the boundedness of the extremal solution in dimensions $n<10s$.
Although this result is optimal as $s\rightarrow 1$, it is not optimal, however, for smaller values of $s\in(0,1)$.
More precisely, an argument in \cite{RS-extremal} suggested the possibility that the extremal solution $u^*$ could be bounded in all dimensions $n\leq7$ and for all $s\in(0,1)$.
However, our results in \cite{RS-extremal} did not give any $L^\infty$ estimate uniform in $s$.

The aim of this paper is to obtain better $L^\infty$ estimates for the fractional Gelfand problem \eqref{pb-G} whenever $\Omega$ is even and convex with respect to each coordinate axis.
Our main result, stated next, establishes the boundedness of the extremal solution $u^*$ whenever \eqref{ineq-gammas} holds and, in particular, whenever $n\leq7$ independently of $s\in(0,1)$.
As explained in Remark \ref{stability-exp-G}, we expect this result to be optimal.

\begin{thm}\label{th1-G}
Let $\Omega$ be a bounded smooth domain in $\R^n$ which is, for every $i=1,...,n$, convex in the $x_i$-direction and symmetric with respect to $\{x_i=0\}$.
Let $s\in(0,1)$, and let $u^*$ be the extremal solution of problem \eqref{pb-G}.
Assume that either $n\leq 2s$, or that $n>2s$ and
\begin{equation}\label{ineq-gammas}
\frac{\Gamma\left(\frac{n}{2}\right)\Gamma(1+s)}{\Gamma\left(\frac{n-2s}{2}\right)}> \frac{\Gamma^2\left(\frac{n+2s}{4}\right)}{\Gamma^2\left(\frac{n-2s}{4}\right)}.
\end{equation}
Then, $u^*$ is bounded.
In particular, the extremal solution $u^*$ is bounded for all $s\in(0,1)$ whenever $n\leq7$.
The same holds if $n=8$ and $s\gtrsim0.28206...$, or if $n=9$ and $s\gtrsim0.63237...$.
\end{thm}

The result is new even in the unit ball $\Omega=B_1$.

We point out that, for $n=10$ condition \eqref{ineq-gammas} is equivalent to $s>1$.

Let us next comment on some works related to problem \eqref{pb-G}.

On the one hand, for the power nonlinearity $f(u)=(1+u)^p$, $p>1$, the problem has been recently studied by D\'avila-Dupaigne-Wei \cite{DDW-forthcoming}.
Their powerful methods are based on a monotonicity formula and a blow-up argument, using the ideas introduced in \cite{DDWW} to study the case of the bilaplacian, $s=2$.
For this case $s=2$, extremal solutions with exponential nonlinearity have been also studied; see for example \cite{DDGM}.

On the other hand, Capella-D\'avila-Dupaigne-Sire \cite{CDDS} studied the extremal solution in the unit ball for general nonlinearities for a related operator but different than the fractional Laplacian \eqref{laps}.
More precisely, they considered the \emph{spectral} fractional Laplacian in $B_1$, i.e., the operator $A^s$ defined via the Dirichlet eigenvalues of the Laplacian in $B_1$.
They obtained an $L^\infty$ bound for $u^*$ in dimensions $n<2\left(2+s+\sqrt{2s+2}\right)$ and, in particular, their result yields the boundedness of the extremal solution in dimensions $n\leq 6$ for all $s\in(0,1)$.

Another result in a similar direction is \cite{DDM}, where D\'avila-Dupaigne-Montenegro studied the extremal solution of a boundary reaction problem.
Recall that problems of the form \eqref{pb-G} involving the fractional Laplacian can be seen as a local weighted problem in $\R^{n+1}_+$ by using the extension of Caffarelli-Silvestre.
Similarly, the spectral fractional Laplacian $A^s$ can be written in terms of an extension in $\Omega\times\R_+$.
Thus, the boundary reaction problem studied in \cite{DDM} is also related to a ``fractional'' problem on the boundary, in which $s=1/2$.
Although in this paper we never use the extension problem for the fractional Laplacian, we will use some ideas appearing in \cite{DDM} to prove our results, as explained next.

Recall that the key property of the extremal solution $u^*$ is that it is \emph{stable} \cite{D,RS-extremal}, in the sense that
\[\int_{\Omega}\lambda e^{u^*}\eta^2dx\leq \int_{\R^n}\left|(-\Delta)^{s/2}\eta\right|^2dx\]
for all $\eta\in H^s(\R^n)$ satisfying $\eta\equiv0$ in $\R^n\setminus \Omega$.

In the classical case $s=1$, the main idea of the proof in \cite{CR} is to take $\eta= e^{\alpha u^*}-1$ in the stability condition to obtain a $W^{2,p}$ bound for $u^*$.
When $n<10$, this $W^{2,p}$ estimate leads, by the Sobolev embeddings, to the boundedness of $u^*$.
This is also the approach that we followed in \cite{RS-extremal} to obtain regularity in dimensions $n<10s$.

Here, instead, we assume by contradiction that $u^*$ is singular, and we prove a lower bound for $u^*$ near its singular point.
This is why we need to assume the domain $\Omega$ to be even and convex ---in this case, the singular point is necessarily the origin.
Then, in the stability condition we take an explicit function $\eta(x)$ with the same expected singular behavior as $e^{\alpha u^*(x)}$ (given by the previous lower bound).
More precisely, we take as $\eta$ a power function of the form $\eta(x)\sim |x|^{-\beta}$, with $\beta$ chosen appropriately.
This idea was already used in \cite{DDM}, where D\'avila-Dupaigne-Montenegro studied the extremal solution for a boundary reaction problem.

The paper is organized as follows.
First, in Section \ref{sec-remarks-G} we give some remarks and preliminary results that will be used in the proof of our main result.
Then, in Section~\ref{sec2-G} we prove Theorem \ref{th1-G}.

\section{Some preliminaries and remarks}
\label{sec-remarks-G}

In this section we recall some facts that will be used in the proof of Theorem \ref{th1-G}.

First, recall that a weak solution $u$ of \eqref{pb-G} is said to be stable when
\begin{equation}\label{semistable2-G}
\int_{\Omega}\lambda e^u\eta^2dx\leq \int_{\R^n}\left|(-\Delta)^{s/2}\eta\right|^2dx
\end{equation}
for all $\eta\in H^s(\R^n)$ satisfying $\eta\equiv0$ in $\R^n\setminus\Omega$; see \cite{RS-extremal} for more details.
Note also that, integrating by parts on the right hand side, one can write \eqref{semistable2-G} as
\begin{equation}\label{semistable3-G}
\int_{\Omega}\lambda e^u\eta^2dx\leq \int_{\Omega}\eta(-\Delta)^s\eta\,dx.
\end{equation}
We will use this form of the stability condition in the proof of Theorem \ref{th1-G}.

Next we recall a computation done in \cite{RS-extremal} in which we can see that condition \eqref{ineq-gammas} arises naturally.

\begin{prop}[\cite{RS-extremal}]
Let $s\in(0,1)$, $n>2s$, and $u_0(x)=\log\frac{1}{|x|^{2s}}$.
Then, $u_0$ is a solution of
\[(-\Delta)^su_0=\lambda_0e^{u_0}\quad \textrm{in all of}\ \R^n,\]
with
\begin{equation}\label{lambda0-G}
\lambda_0=2^{2s}\frac{\Gamma\left(\frac{n}{2}\right)\Gamma(1+s)}{\Gamma\left(\frac{n-2s}{2}\right)}.
\end{equation}
Moreover, setting
\begin{equation}\label{Hns-G}
H_{n,s}=2^{2s}\frac{\Gamma^2\left(\frac{n+2s}{4}\right)}{\Gamma^2\left(\frac{n-2s}{4}\right)},
\end{equation}
$u_0$ is stable if and only if $\lambda_0\leq H_{n,s}$.
\end{prop}

We point out that $H_{n,s}$ is the best constant in the fractional Hardy inequality, even though we will not use such inequality in this paper.

\begin{rem}\label{stability-exp-G}
This proposition suggests that there could exist a stable singular solution to \eqref{pb-G} in the unit ball whenever $\lambda_0\leq H_{n,s}$.
In fact, we may consider a larger family of problems than \eqref{pb-G}, by considering nonhomogeneous Dirichlet conditions of the form $u=g$ in $\R^n\setminus\Omega$.
For all these problems, our result in Theorem \ref{th1-G} still remains true; see Remark \ref{rem-G}.
In the particular case $\Omega=B_1$ and $g(x)=\log |x|^{-2s}$ in $\R^n\setminus B_1$, the extremal solution to the new problem is exactly $u^*(x)=\log |x|^{-2s}$ in $B_1$ whenever $\lambda_0\leq H_{n,s}$.
Thus, when $\lambda_0\leq H_{n,s}$ we have a singular extremal solution for some exterior condition $g$.

We expect the sufficient condition \eqref{ineq-gammas} of Theorem \ref{th1-G} to be optimal since it is equivalent to $\lambda_0>H_{n,s}$.

The condition $\lambda_0>H_{n,s}$, appeared and was discussed in Remark 3.3 in \cite{RS-extremal}.
\end{rem}

We next give a symmetry result, which is the analog of the classical result of Berestycki-Nirenberg \cite{BN}.
It does not require any smoothness of $\Omega$.
From this result it will follow that, under the hypotheses of Theorem \ref{th1-G}, the solutions $u_\lambda(x)$ attain its maxima at $x=0$.

When $\Omega=B_R$, there are a number of papers proving the radial symmetry of solutions for nonlocal equations.
For domains which are symmetric with respect to an hyperplane, the following Lemma follows from the results in \cite{Jarohs}.

\begin{lem}[\cite{Jarohs}]\label{lem-symmetry}
Let $\Omega$ be a bounded domain which is convex in the $x_1$-direction and symmetric with respect to $\{x_1=0\}$.
Let $f$ be a locally Lipschitz function, and $u$ be a bounded positive solution of
\[\left\{ \begin{array}{rcll}
(-\Delta)^s u &=&f(u)&\textrm{in }\Omega \\
u&=&0&\textrm{in }\mathbb R^n\backslash \Omega.
\end{array}\right.\]
Then, $u$ is symmetric with respect to $\{x_1=0\}$, and it satisfies
\[\partial_{x_1}u<0\quad \textrm{in}\quad \Omega\cap \{x_1>0\}.\]
\end{lem}

As said before, this lemma yields that solutions $u_\lambda$ of \eqref{pb-G} satisfy
\[\|u_\lambda\|_{L^\infty(\Omega)}=u_\lambda(0).\]
This allows us to locate the (possible) singularity of the extremal solution $u^*$ at the origin, something that is essential in our proofs.

Finally, to end this section, we compute the fractional Laplacian on a power function, something needed in the proof of Theorem \ref{th1-G}.

\begin{prop}\label{powers-G}
Let $(-\Delta)^s$ be the fractional Laplacian in $\R^n$, with $s>0$ and $n>2s$.
Let $\alpha\in(0,n-2s)$.
Then,
\[(-\Delta)^s |x|^{-\alpha}=2^{2s}\,\frac{\Gamma\left(\frac{\alpha+2s}{2}\right)\Gamma\left(\frac{n-\alpha}{2}\right)}
{\Gamma\left(\frac{n-\alpha-2s}{2}\right)\Gamma\left(\frac{\alpha}{2}\right)}\,|x|^{-\alpha-2s},\]
where $\Gamma$ is the Gamma function.
\end{prop}

\begin{proof}
We use Fourier transform, defined by
\[\mathcal F[u](\xi)=(2\pi)^{-n/2}\int_{\R^n}u(x)e^{-i\xi\cdot x}dx.\]
Then, one has that
\begin{equation}\label{Ftrans-laps-G}
\mathcal F\bigl[(-\Delta)^su\bigr](\xi)=|\xi|^{2s}\,\mathcal F[u](\xi).
\end{equation}

On the other hand, the function $|x|^{-\alpha}$, with $0<\alpha<n$, has Fourier transform
\begin{equation}\label{Ftrans-G}
\kappa_\beta\,\mathcal F\bigl[|\cdot|^{-\beta}\bigr](\xi)=\kappa_{n-\beta}|\xi|^{\beta-n},\qquad \kappa_\beta:=2^{\beta/2}\Gamma(\beta/2)
\end{equation}
(see for example \cite[Theorem 5.9]{Lieb-Loss}, where another convention for the Fourier transform is used, however).

Hence, using \eqref{Ftrans-G} and \eqref{Ftrans-laps-G}, we find that
\[\begin{split}
\mathcal F\bigl[(-\Delta)^s|\cdot|^{-\alpha}\bigr](\xi)&=|\xi|^{2s}\,\mathcal F\bigl[|\cdot|^{-\alpha}\bigr](\xi)\\&= \frac{\kappa_{n-\alpha}}{\kappa_{\alpha}}\,|\xi|^{\alpha+2s-n}
=\frac{\kappa_{n-\alpha}}{\kappa_{\alpha}}\,\frac{\kappa_{\alpha+2s}}{\kappa_{n-\alpha-2s}}\,\mathcal F\bigl[|\cdot|^{-\alpha-2s}\bigr](\xi).
\end{split}\]
Thus, it follows that
\[(-\Delta)^s|x|^{-\alpha}=\frac{\kappa_{n-\alpha}}{\kappa_{\alpha}}\,\frac{\kappa_{\alpha+2s}}{\kappa_{n-\alpha-2s}}\,|x|^{-\alpha-2s}=
2^{2s}\,\frac{\Gamma\left(\frac{\alpha+2s}{2}\right)\Gamma\left(\frac{n-\alpha}{2}\right)}
{\Gamma\left(\frac{n-\alpha-2s}{2}\right)\Gamma\left(\frac{\alpha}{2}\right)}\,|x|^{-\alpha-2s},\]
as claimed.
\end{proof}

\section{Proof of the main result}
\label{sec2-G}

The aim of this section is to prove Theorem \ref{th1-G}.
We start with two preliminary lemmas.

The first one gives a lower bound for the singularity of an unbounded extremal solution.
As we will see, this is an essential ingredient in our proof of Theorem \ref{th1-G}.
A similar result was established in \cite{DDM} in the case of the boundary reaction problem considered there.

\begin{lem}\label{lem-essential-G}
Let $n$, $s$, and $u^*$ as in Theorem \ref{th1-G}, and assume that $u^*$ is unbounded.
Then, for each $\sigma\in(0,1)$ there exists $r(\sigma)>0$ such that
\[u^*(x)>(1-\sigma)\log \frac{1}{|x|^{2s}}\]
for all $x$ satisfying $|x|<r(\sigma)$.
\end{lem}

\begin{proof}
We will argue by contradiction.
Assume that there exist $\sigma\in(0,1)$ and a sequence $\{x_k\}\rightarrow0$ for which
\begin{equation}\label{contrad}
u^*(x_k)\leq (1-\sigma)\log\frac{1}{|x_k|^{2s}}.
\end{equation}
Recall that, by Lemma \ref{lem-symmetry}, we have $u_\lambda(0)=\|u_\lambda\|_{L^\infty}$.
Thus, since $u^*$ is unbounded by assumption, we have
\[\|u_\lambda\|_{L^\infty(\Omega)}=u_\lambda(0)\longrightarrow +\infty\quad \textrm{as}\quad \lambda\rightarrow\lambda^*.\]
In particular, there exists a sequence $\{\lambda_k\}\rightarrow\lambda^*$ such that
\[u_{\lambda_k}(0)=\log\frac{1}{|x_k|^{2s}}.\]

Define now the functions
\[v_k(x)=\frac{u_{\lambda_k}(|x_k|x)}{\|u_{\lambda_k}\|_{L^\infty}}=\frac{u_{\lambda_k}(|x_k|x)}{\log\frac{1}{|x_k|^{2s}}},\qquad x\in \Omega_k=\frac{1}{|x_k|}\Omega.\]
These functions satisfy $0\leq v_k\leq 1$, $v_k(0)=1$, and
\[(-\Delta)^s v_k\longrightarrow0\quad \textrm{uniformly in}\ \Omega_k\ \textrm{as}\  k\rightarrow\infty.\]
Indeed,
\[(-\Delta)^s v_k(x)=\frac{1}{\log\frac{1}{|x_k|^{2s}}}|x_k|^{2s}\lambda_ke^{u_{\lambda_k}(|x_k|x)}\leq \frac{\lambda_k}{\log\frac{1}{|x_k|^{2s}}}\leq \frac{\lambda^*}{\log\frac{1}{|x_k|^{2s}}}\longrightarrow 0.\]

Note also that the functions $v_k$ are uniformly H\"older continuous in compact sets of $\R^n$, since $|(-\Delta)^sv_k|$ are uniformly bounded (see for example Proposition 1.1 in \cite{RS-Dir}).
Hence, it follows from the Arzel\`a-Ascoli theorem that, up to a subsequence, $v_k$ converges uniformly in compact sets of $\R^n$ to some function $v$ satisfying
\[(-\Delta)^s v\equiv0\quad \textrm{in}\ \R^n,\quad 0\leq v\leq 1,\quad v(0)=1.\]
Thus, it follows from the strong maximum principle that $v\equiv 1$.

Therefore, we have that
\[v_k(x)\longrightarrow 1 \quad \textrm{uniformly in compact sets of}\ \R^n,\]
and in particular
\[\frac{u_{\lambda_k}(x_k)}{\log\frac{1}{|x_k|^{2s}}}=v_k\left(x_k/|x_k|\right)\longrightarrow 1.\]
This contradicts \eqref{contrad}, and hence the lemma is proved.
\end{proof}

In the next lemma we compute the fractional Laplacian of some explicit functions in all of $\R^n$.
The constants appearing in these computations are very important, since they are very related to the ones in \eqref{ineq-gammas}.

\begin{lem}\label{expl-comp-G}
Let $s\in(0,1)$, $n>2s$, and $\varepsilon>0$ be small enough.
Then
\[(-\Delta)^s |x|^{\frac{2s-n+\varepsilon}{2}}=\left(H_{n,s}+O(\varepsilon)\right)|x|^{\frac{-2s-n+\varepsilon}{2}}\]
and
\[(-\Delta)^s |x|^{2s-n+\varepsilon}=\left(\lambda_0\frac{\varepsilon}{2s}+O(\varepsilon^2)\right)|x|^{-n+\varepsilon},\]
where $H_{n,s}$ and $\lambda_0$ are given by \eqref{Hns-G} and \eqref{lambda0-G}, respectively.
\end{lem}

\begin{proof}
To prove the result we use Proposition \ref{powers-G} and the properties of the $\Gamma$ function, as follows.

First, using Proposition \ref{powers-G} with $\alpha=\frac12(n-2s-\varepsilon)$ and with $\alpha=n-2s-\varepsilon$, we find
\[(-\Delta)^s |x|^{\frac{2s-n+\varepsilon}{2}}=2^{2s}
\frac{\Gamma\left(\frac{n+2s-\varepsilon}{4}\right)\Gamma\left(\frac{n+2s+\varepsilon}{4}\right)}
{\Gamma\left(\frac{n-2s+\varepsilon}{4}\right)\Gamma\left(\frac{n-2s-\varepsilon}{4}\right)}\,|x|^{\frac{-2s-n+\varepsilon}{2}}\]
and
\[(-\Delta)^s |x|^{2s-n+\varepsilon}=2^{2s}\frac{\Gamma\left(\frac{n-\varepsilon}{2}\right)\Gamma\left(\frac{2s+\varepsilon}{2}\right)}
{\Gamma\left(\frac{\varepsilon}{2}\right)\Gamma\left(\frac{n-2s-\varepsilon}{2}\right)}\,|x|^{-n+\varepsilon},\]
where $\Gamma$ is the Gamma function.

Since $\Gamma(t)$ is smooth and positive for $t>0$, then it is clear that
\[2^{2s}
\frac{\Gamma\left(\frac{n+2s-\varepsilon}{4}\right)\Gamma\left(\frac{n+2s+\varepsilon}{4}\right)}
{\Gamma\left(\frac{n-2s+\varepsilon}{4}\right)\Gamma\left(\frac{n-2s-\varepsilon}{4}\right)}=
2^{2s}\left(\frac{\Gamma\left(\frac{n+2s}{4}\right)}
{\Gamma\left(\frac{n-2s}{4}\right)}\right)^2+O(\varepsilon)=H_{n,s}+O(\varepsilon).\]
Thus, the first identity of the Lemma follows.

To prove the second identity, we use also that $\Gamma(1+t)=t\Gamma(t)$.
We find,
\[\begin{split}
2^{2s}\frac{\Gamma\left(\frac{n-\varepsilon}{2}\right)\Gamma\left(\frac{2s+\varepsilon}{2}\right)}
{\Gamma\left(\frac{\varepsilon}{2}\right)\Gamma\left(\frac{n-2s-\varepsilon}{2}\right)}
&=
2^{2s}\frac{\Gamma\left(\frac{n-\varepsilon}{2}\right)\Gamma\left(\frac{2s+\varepsilon}{2}\right)}
{\Gamma\left(1+\frac{\varepsilon}{2}\right)\Gamma\left(\frac{n-2s-\varepsilon}{2}\right)}\,\frac{\varepsilon}{2}\\
&=
2^{2s}\frac{\Gamma\left(\frac{n}{2}\right)\Gamma(s)}
{\Gamma(1)\Gamma\left(\frac{n-2s}{2}\right)}\left(1+O(\varepsilon)\right)\frac{\varepsilon}{2}\\
&= 2^{2s}\frac{\Gamma\left(\frac{n}{2}\right)s\Gamma(s)}
{\Gamma\left(\frac{n-2s}{2}\right)}\left(\frac{\varepsilon}{2s}+O(\varepsilon^2)\right)\\
&=2^{2s}\frac{\Gamma\left(\frac{n}{2}\right)\Gamma(1+s)}
{\Gamma\left(\frac{n-2s}{2}\right)}\left(\frac{\varepsilon}{2s}+O(\varepsilon^2)\right)\\
&=\frac{\lambda_0}{2s}\,\varepsilon+O(\varepsilon^2).
\end{split}\]
Thus, the lemma is proved.
\end{proof}

We can now give the proof of our main result.

\begin{proof}[Proof of Theorem \ref{th1-G}]
First, note that when $n\leq2s$ the result follows from \cite{RS-extremal}, since we proved there the result for $n<10s$.
Thus, from now on we assume $n>2s$.

To prove the result for $n>2s$ we argue by contradiction, that is, we assume that $u^*$ is unbounded and we show that this yields $\lambda_0\leq  H_{n,s}$.
As we will see, Lemma \ref{lem-essential-G} plays a very important role in this proof.

Let $u_\lambda$, with $\lambda<\lambda^*$, be the minimal stable solution to \eqref{pb-G}.
Using $\psi$ in the stability condition \eqref{semistable3-G}, we obtain
\[\int_{\Omega}\lambda e^{u_\lambda}\psi^2dx\leq \int_{\Omega}\psi(-\Delta)^s\psi\,dx.\]
Moreover, $\psi^2$ as a test function for the equation \eqref{pb-G}, we find
\[\int_{\Omega}u_\lambda(-\Delta)^s(\psi^2)dx=\int_{\Omega}\lambda e^{u_\lambda}\psi^2dx.\]
Thus, we have
\begin{equation}\label{ineq-G}
\int_{\Omega}u_\lambda(-\Delta)^s(\psi^2)dx\leq \int_{\Omega}\psi(-\Delta)^s\psi\,dx\quad \textrm{for all}\quad \lambda<\lambda^*.
\end{equation}

Next we choose $\psi$ appropriately so that \eqref{ineq-G} combined with Lemma \ref{lem-essential-G} yield a contradiction.
This function $\psi$ will be essentially a power function $|x|^{-\beta}$, as explained in the Introduction.

Indeed, let $\rho_0$ be small enough so that $B_{\rho_0}(0)\subset \Omega$.
For each small $\varepsilon>0$, let us consider a function $\psi$ satisfying
\begin{enumerate}
\item $\psi(x)=|x|^{\frac{2s-n+\varepsilon}{2}}$ in $B_{\rho_0}(0)\subset \Omega$.
\item $\psi$ has compact support in $\Omega$.
\item $\psi$ is smooth in $\R^n\setminus\{0\}$.
\end{enumerate}

Now, since the differences $\psi(x)-|x|^{\frac{2s-n+\varepsilon}{2}}$ and $\psi^2(x)-|x|^{2s-n+\varepsilon}$ are smooth and bounded in all of $\R^n$ (by definition of $\psi$), then it follows from Lemma \ref{expl-comp-G} that
\begin{equation}\label{cosa1-G}
(-\Delta)^s\psi(x)\leq \left(H_{n,s}+C\varepsilon\right)|x|^{\frac{-2s-n+\varepsilon}{2}}+C
\end{equation}
and
\begin{equation}\label{cosa2-G}
(-\Delta)^s (\psi^2)(x)\geq \left(\lambda_0\frac{\varepsilon}{2s}-C\varepsilon^2\right)|x|^{-n+\varepsilon}-C,
\end{equation}
where $C$ is a constant that depends on $\rho_0$ but not on $\varepsilon$.

In the rest of the proof, $C$ will denote different constants, which may depend on $\rho_0$, $n$, $s$, $\Omega$, and $\sigma$, but not on $\varepsilon$.
Here, $\sigma$ is any given number in $(0,1)$.

Hence, we deduce from \eqref{ineq-G}-\eqref{cosa1-G}-\eqref{cosa2-G}, that
\begin{equation}\label{tgh-G}
\left(\lambda_0\frac{\varepsilon}{2s}-C\varepsilon^2\right)\int_{\Omega}u_\lambda|x|^{\varepsilon-n}dx\leq \left(H_{n,s}+C\varepsilon\right)\int_{\Omega}|x|^{\varepsilon-n}dx+C.
\end{equation}
We have used that $\int_\Omega u_\lambda\leq \int_\Omega u^*\leq C$ uniformly in $\lambda$ (see \cite{RS-extremal}).
Since the right hand side does not depend on $\lambda$, we can let $\lambda\longrightarrow\lambda^*$ to find that \eqref{tgh-G} holds also for $\lambda=\lambda^*$.

Next, for the given $\sigma\in(0,1)$, we apply Lemma \ref{lem-essential-G}.
Since $u^*$ is unbounded by assumption, we deduce that there exists $r(\sigma)>0$ such that
\[u^*(x)\geq (1-\sigma)\log\frac{1}{|x|^{2s}}\quad \textrm{in}\  B_{r(\sigma)}.\]
Thus, we find
\begin{equation}\label{find-G}
(1-\sigma) \left(\lambda_0\frac{\varepsilon}{2s}-C\varepsilon^2\right)\int_{B_{r(\sigma)}}|x|^{\varepsilon-n}\log\frac{1}{|x|^{2s}}dx\leq  \left(H_{n,s}+C\varepsilon\right)\int_{\Omega}|x|^{\varepsilon-n}dx+C.
\end{equation}

Now, we have
\[\begin{split}
\int_{B_{r(\sigma)}}|x|^{\varepsilon-n}\log\frac{1}{|x|^{2s}}dx&=2s|S^{n-1}|\int_0^{r(\sigma)}r^{\varepsilon-1}\log\frac{1}{r}\,dr \\
&=2s|S^{n-1}|\, \left(r(\sigma)\right)^\varepsilon\frac{1-\varepsilon\log\frac{1}{r(\sigma)}}{\varepsilon^2}\\
&\geq \left\{2s|S^{n-1}|\left(r(\sigma)\right)^\varepsilon-C\varepsilon\right\}\frac{1}{\varepsilon^2}\end{split}\]
and
\[\int_{\Omega}|x|^{\varepsilon-n}dx\leq |S^{n-1}|\int_0^1r^{\varepsilon-1}dr+C=|S^{n-1}|\frac{1}{\varepsilon}+C.\]
Therefore, by \eqref{find-G},
\[(1-\sigma) \left(\lambda_0\frac{\varepsilon}{2s}-C\varepsilon^2\right) \left\{2s|S^{n-1}|\left(r(\sigma)\right)^\varepsilon-C\varepsilon\right\}\frac{1}{\varepsilon^2}\leq  \left(H_{n,s}+C\varepsilon\right)|S^{n-1}|\frac{1}{\varepsilon}+C.\]
Hence, multiplying by $\varepsilon$ and rearranging terms,
\[(1-\sigma)\lambda_0\left(r(\sigma)\right)^\varepsilon \leq  H_{n,s}+C\varepsilon.\]
Letting now $\varepsilon\rightarrow0$ (recall that $\sigma\in(0,1)$ is an arbitrary given number), we find
\[(1-\sigma) \lambda_0\leq  H_{n,s}.\]

Finally, since this can be done for each $\sigma\in(0,1)$, we deduce that
\[\lambda_0\leq  H_{n,s},\]
a contradiction.
\end{proof}

\begin{rem}\label{rem-G}
Note that in our proof of Theorem \ref{th1-G} the exterior condition $u\equiv0$ in $\R^n\setminus\Omega$ plays no role.
Thus, the same result holds true for \eqref{pb-G} with any other exterior condition $u=g$ in $\R^n\setminus\Omega$.

On the other hand, note that the nonlinearity $f(u)=e^u$ plays a very important role in our proof.
Indeed, to establish \eqref{ineq-G} we have strongly used that $f'(u)=f(u)$, since we combined the stability condition (in which $f'(u)$ appears) with the equation (in which only $f(u)$ appears).
It seems difficult to extend our proof to the case of more general nonlinearities.
Even for the powers $f(u)=(1+u)^p$, it is not clear how to do it.

\end{rem}

\section*{Acknowledgements}

The author thanks Xavier Cabr\'e for his guidance and useful discussions on the topic of this paper.

\end{document}